\documentclass[amstex,11pt,reqno]{amsart}
\usepackage{amsmath,amsfonts,amssymb,amsthm,hyperref,enumerate,multicol,graphicx}
\newtheorem{theorem}{Theorem}[section]
\newtheorem{lemma}[theorem]{Lemma}

\theoremstyle{definition}

\theoremstyle{remark}

\numberwithin{equation}{section}

\begin{document}
\setcounter{page}{1}

\title[$\lambda $-Bernstein operators]{Some general statistical approximation results for $ \lambda $-Bernstein operators}
\author[F. \"{O}zger]{Faruk \"{O}zger}
\maketitle
\begin{center}
  \address{Department of Engineering Sciences, \.{I}zmir Katip \c{C}elebi University, 35620, \.{I}zmir, Turkey}\newline

\email{farukozger@gmail.com}
\end{center}

\begin{abstract}
In this article, we achieve some general statistical approximation results for $ \lambda $-Bernstein operators in addition to some other approximation properties. We prove a statistical Voronovskaja-type approximation theorem. We also construct bivariate $\lambda $-Bernstein operators and study their approximation properties.
\end{abstract}

\section{Introduction}

Bernstein used famous polynomials nowadays called Bernstein polynomials, in 1912, to obtain an alternative proof of Weierstrass's fundamental theorem \cite{bernstein}.  Approximation properties of Bernstein operators and their applications in Computer Aided Geometric Design and Computer Graphics have been extensively studied in many articles.

Bernstein basis of degree $n$ on $x \in [0, 1]$ is defined by
$$b_{n,i}(x)=\binom{n}{i}x^i~(1-x)^{n-i}\qquad~\qquad~i=0,\dots, n,$$
and $n$th order Bernstein polynomial
is given by
\begin{eqnarray}\label{46}
B_n(f; x)=\sum_{i=0}^nf\left(\frac{i}{n}\right)~b_{n, i}(x)
\end{eqnarray}%
for any continuous function $f(x)$ defined on $[0, 1]$.

In 2018, Cai et al. have introduced a new type $ \lambda $ Bernstein operators \cite{cai} 
\begin{eqnarray}\label{caioperator}
B_{n,\lambda}(f; x)=\sum_{i=0}^nf\left(\frac{i}{n}\right)~\tilde{b}_{n, i}(\lambda;x)
\end{eqnarray}%
with B\'{e}zier bases $ \tilde{b}_{n, i}(\lambda;x) $ \cite{ye}:
\begin{align}
\left\{\begin{array}{ccl}
&&\tilde{b}_{n, 0}(\lambda; x)=b_{n, 0}(x)-\dfrac{\lambda}{n+1}b_{n+1, 1} (x),\\
&&\tilde{b}_{n, i}(\lambda; x)=b_{n, i}(x)+\lambda \left( \dfrac{n-2i+1}{n^2-1}b_{n+1, i}(x)-\dfrac{n-2i-1}{n^2-1}b_{n+1, i+1}(x)\right) ,~~i=1, 2 \dots, n-1,\\
&&\tilde{b}_{n ,n}(\lambda; x)=b_{n, n}(x)-\dfrac{\lambda}{n+1}b_{n+1, n} (x),\end{array} \right. \label{bezier}
\end{align}
where shape parameters $\lambda \in [-1,1]$.

\section{Preliminary Results}
In this part, we obtain global approximation formula in terms of Ditzian-Totik uniform modulus of smoothness of first and second order and give a local direct estimate of the rate of convergence by Lipschitz-type function involving two parameters for $ \lambda $ Bernstein operators. We also give a Gr\"{u}ss–-Voronovskaja and a quantitative Voronovskaja-type theorem

Results in the following lemma were obtained for $ \lambda $ Bernstein operators in \cite[Lemma 2.1]{cai}.
\begin{lemma} We have following equalities for $ \lambda $ Bernstein operators:
	\begin{align}\label{lemmamoment}
	B_{n,\lambda}(1; x)&=1;\\
	\nonumber B_{n,\lambda}(t; x)&=x+\frac{1-2x+x^{n+1}-(1-x)^{n+1}}{n(n-1)}\lambda;\\
	\nonumber	B_{n,\lambda}(t^2; x)&=x^2+\frac{x(1-x)}{n}+\left[ \frac{2x-4x^2+2x^{n+1}}{n(n-1)}+ \frac{x^{n+1}+(1-x)^{n+1}-1}{n^2(n-1)}  \right] \lambda;\\
	\nonumber	 B_{n,\lambda}(t^3; x)&=x^3+\frac{3x^2(1-x)}{n}+\frac{2x^3-3x^2+x}{n^2}+\left[ \frac{6x^{n+1}-6x^3}{n^2}+ \frac{3x^2-3x^{n+1}}{n(n-1)}\right. \\
	\nonumber	 &\quad+\frac{9x^{n+1}-9x^2}{n^2(n-1)}+ \frac{4x^{n+1}-4x}{n^3(n-1)}+\left. \frac{1-x^{n+1}+(1-x)^{n+1}}{n^3(n^2-1)}\right] \lambda;\\
	\nonumber	B_{n,\lambda}(t^4; x)&=x^4+\frac{6x^3(1-x)}{n}+\frac{7x^2-18x^3+11x^4}{n^2}+\frac{x-7x^2+12x^3-6x^4}{n^3}\\
	\nonumber	&\quad+\left[ \frac{6x^2-2x^3-8x^4+4x^{n+1}}{n^2}+ \frac{17x^{n+1}+16x^4-32x^3-x^2}{n^3}+\frac{x-x^{n+1}}{n^4}\right.\\
	\nonumber	&\quad+\left. \frac{7x^2-7x^{n+1}}{n^2(n-1)}+\frac{x-23x^2+22x^{n+1}}{n^3(n-1)}+\frac{(1-x)^{n+1}+x-1}{n^4(n-1)}\right] \lambda.
	\end{align}
\end{lemma}
\subsection{Global and local approximations}
First we obtain global approximation formula in terms of Ditzian-Totik uniform modulus of smoothness of first and second order defined by
\begin{eqnarray*}
	\omega_{\xi}(f, \delta):=\sup_{0<|h|\leq\delta} ~\sup_{x, x+ h\xi(x) \in [0,1]}\{|f(x+h\xi(x))-f(x)|\}
\end{eqnarray*}
and
\begin{eqnarray*}
	\omega_2^{\phi}(f, \delta):=\sup_{0<|h|\leq\delta} ~\sup_{x,x\pm h\phi(x)\in [0,1]}\{|f(x+h\phi(x))-2f(x)+f(x-h\phi(x))|\},
\end{eqnarray*}
respectively, where $\phi$ is an admissible step-weight function on $[a, b]$, i.e. $\phi(x)=[(x-a)(b-x)]^{1/2}$ if $x \in [a, b]$, \cite{zdi}. Corresponding $K$-functional is 
\begin{eqnarray*}
	K_{2, \phi(x)}(f,\delta)=\inf_{g \in W^2(\phi)}\big\{||f-g||_{C[0,1]}+\delta||\phi^2g''||_{C[0,1]}:g\in C^2[0 ,1]\big\},
\end{eqnarray*}
where $\delta>0$, $W^2(\phi)=\{g\in C[0,1]:g' \in AC[0,1],~ \phi^2g'' \in C[0,1]\}$ and $C^2[0 ,1]=\{g\in C[0,1]:g', g'' \in C[0,1]\}$. Here, $g' \in AC[0,1]$ means that $g'$
is absolutely continuous on $[0,1]$. It is known by \cite{DeVore-Lorentz} that there exists an absolute constant $C > 0$, such that
\begin{eqnarray}\label{3000}
C^{-1}\omega_2^{\phi}(f, \sqrt{\delta})\leq K_{2, \phi(x)}(f,\delta)\leq C \omega_2^{\phi}(f, \sqrt{\delta}).
\end{eqnarray}

\begin{theorem}\label{alphan} Let $ \lambda \in [-1,1], $ $f \in C[0, 1]$ and $\phi$~$(\phi\neq 0)$ be an admissible step-weight function of Ditzian-Totik modulus of smoothness such that $\phi^2$ is concave. Then we have
	\begin{align*}
	|B_{n,\lambda}(f; x)-f(x)|\leq C \omega_2^{\phi}\bigg(f, \frac{\delta_n(x)}{2\phi(x)}\bigg)+\omega_{\xi}\bigg(f, \frac{\beta_{n}(x)}{\xi(x)}\bigg)
	\end{align*}
for $x \in [0,1]$ and $C>0$. 
\end{theorem}

Now we give a local direct estimate of the rate of convergence with the help of Lipschitz-type function involving two parameters for operators (\ref{caioperator}). We write
\begin{eqnarray*}
	\nonumber Lip^{(k_1, k_2)}_{M}(\eta):=\Big\{f \in C[0, 1]:|f(t)-f(x)|\le M \frac{|t-x|^\eta}{(k_1x^2+k_2x+t)^{\frac{\eta}{2}}};~x \in (0, 1], t\in [0, 1]\Big\}
\end{eqnarray*}
for $k_1\ge 0, k_2>0$, where $\eta \in (0, 1]$ and $M$ is a positive constant (see \cite{lipztwo}).

\begin{theorem} \label{local} If $f \in Lip^{(k_1, k_2)}_{M}(\eta)$, then we have
	\begin{align*}
	|B_{n,\lambda}(f; x)-f(x)|&\leq M\alpha^{\frac{\eta}{2}}_n(x) (k_1x^2+k_2x)^{\frac{-\eta}{2}}
	\end{align*}
for all $ \lambda \in [-1,1] , $ $x \in (0,1]$ and $\eta \in (0, 1]$.
\end{theorem}

\begin{theorem}The following inequality holds:
	\begin{align*}
	|B_{n,\lambda}(f; x)-f(x)|&\leq |\beta_n(x)|~|f'(x)|+2\sqrt{\alpha_n(x)}w\big(f', \sqrt{\alpha_n(x)}~\big)
	\end{align*}
for $f \in C^{1}[0 ,1]$ and $x \in [0 ,1]$.
\end{theorem}

\subsection{Voronovskaja-type theorems}
In this part, we give a Gr\"{u}ss–Voronovskaja-type theorem and a quantitative Voronovskaja-type theorem for $B_{n,\lambda}(f; x)$.

We first obtain a quantitative Voronovskaja-type theorem for $B_{n,\lambda}(f; x)$ using Ditzian-Totik modulus of smoothness defined as 
\begin{eqnarray*}
	\omega_{\phi}(f, \delta):=\sup_{0<|h|\leq\delta} \bigg\{\bigg|f\bigg(x+\frac{h\phi(x)}{2}\bigg)-f\bigg(x-\frac{h\phi(x)}{2}\bigg)\bigg|, x\pm\frac{h\phi(x)}{2}\in [0 ,1]\bigg\},
\end{eqnarray*}
where $\phi(x)=(x(1-x))^{1/2}$ and $f \in C[0 ,1]$, and corresponding Peetre's $K$-functional is defined by 
\begin{eqnarray*}
	K_{\phi}(f,\delta)=\inf_{g \in W_\phi[0,1]}\big\{||f-g||+\delta||\phi g'||:g\in C^1[0 ,1], \delta>0\big\},
\end{eqnarray*}
where $W_{\phi}[0,1]=\{g:~g \in AC_{loc}[0, 1],~\|\phi g'\|<\infty\}$ and $AC_{loc}[0 ,1]$ is the class of absolutely continuous functions defined on $[a, b] \subset [0 ,1]$. There exists a constant $C > 0$ such that
\begin{eqnarray*}
	K_{\phi}(f,\delta)\leq  C~\omega_{\phi}(f, \delta).
\end{eqnarray*}

\begin{theorem} \label{thm10} Assume that $f\in C[0, 1]$ such that $f', f'' \in C[0,1]$. Then, we have
	\begin{align*}
	\bigg|B_{n,\lambda}(f; x)-f(x)-\beta_{n}f'(x)-\Big(\frac{\alpha_n+1}{2}\Big)f''(x)\bigg|\leq \frac{C}{n}\phi^{2}(x) \omega_{\phi}\bigg(f'',n^{-1/2}\bigg)
	\end{align*}
for every $x \in [0 ,1]$ and sufficiently
large $n$, where $C$ is a positive constant, $\alpha_n$ and $\beta_n$ are defined in Theorem \ref{alphan}.
\end{theorem}

\section{Statistical approximation properties by weighted mean matrix method}
In this part, we study on statistical approximation properties and estimate rate of weighted $A$-statistical convergence. We also use statistical convergence to prove a Voronovskaja-type approximation theorem.

\begin{theorem}\label{thm6} Let $A=(a_{nk})$ be a weighted non-negative regular summability matrix for $ n, k\in \mathbb N $ and $q=(q_n)$ be a sequence of non-negative numbers such that $q_0>0$ and $Q_n=\sum_{k=0}^{n}q_k \to \infty$ as $n \to \infty$. 
	For any $f \in C[0 ,1]$, we have
	\begin{align*}
	S_A^{\widetilde N}-\lim_{n\to\infty} \|B_{n,\lambda}(f; x)-f(x)\|_{C[0 ,1]}=0.
	\end{align*}
\end{theorem}

\subsection{A Voronovskaja-type approximation theorem }
We prove a Voronovskaja-type approximation theorem by $\mathring{B}_{n,\lambda}(f; x)$ family of linear operators.

\begin{theorem} Let $A=(a_{nk})$ be a weighted non-negative regular summability matrix and let $(x_n)$
	be a sequence of real numbers such that $S_A^{\widetilde N}-\lim x_n=0$. Also let $\mathring{B}_{n,\lambda}(f; x)$ be a sequence of positive linear operators acting from $C_B[0 ,1]$ into $C[0 ,1]$ defined by $$\mathring{B}_{n,\lambda}(f; x)=(1+x_n)B_{n,\lambda}(f; x).$$ 
	Then for every $f\in C_B[0, 1]$, and $f', f'' \in C_B[0,1]$ we have
	\begin{align*}
	S_A^{\widetilde N}-\lim_{n \to \infty} n\big\{\mathring{B}_{n,\lambda}(f; x)-f(x)\big\}=\frac{f''(x)}{2}x(1-x).
	\end{align*}
\end{theorem}

\section{Approximation properties for bivariate case}
 In this part, we construct bivariate $ \lambda $ Bernstein operators and study their approximation properties.

Let $I=I_1 \times I_1  =[0,1]\times[0,1] $ and $(x,y)\in I  $, then we construct bivariate $ \lambda $ Bernstein operators as
\begin{equation*} \bar{B}_{n,m}\left(f;x,y;\lambda\right) =\sum_{k_1=0}^{n}\sum_{k_2=0}^{m}f\left(\dfrac{k_1}{n},\dfrac{k_2}{m} \right) \tilde{b}_{n,k_1}(\lambda;x)\tilde{b}_{m,k_2}(\lambda;y)
\end{equation*}
for $f\in C(I)  $, where B\'{e}zier bases $\tilde{b}_{n,k_1}(\lambda; x),~~\tilde{b}_{m,k_2}(\lambda; x)~~(k_1=0,1,\dots, n;~~k_2=0,1,\dots, m)$ are defined in (\ref{bezier}).
\begin{lemma}\label{lemx} For any natural number $n$ $(n \ge 2)$ the following equalities hold:
\begin{align}
 \bar{B}_{n,m}(1; x,y;\lambda)&=1;\label{b1}\\
\bar{B}_{n,m}(s; x,y;\lambda)&=x+\frac{1-2x+x^{n+1}-(1-x)^{n+1}}{n(n-1)}\lambda;\label{b2}\\
\bar{B}_{n,m}(t; x,y;\lambda)&=y+\frac{1-2y+y^{m+1}-(1-y)^{m+1}}{m(m-1)}\lambda;\label{b3}\\
\bar{B}_{n,m}(s^2; x,y;\lambda)&=x^2+\frac{x(1-x)}{n}+\left[ \frac{2x-4x^2+2x^{n+1}}{n(n-1)}+ \frac{x^{n+1}+(1-x)^{n+1}-1}{n^2(n-1)}  \right] \lambda;\label{b4}\\
\bar{B}_{n,m}(t^2; x,y;\lambda)&=y^2+\frac{y(1-y)}{m}+\left[ \frac{2y-4y^2+2y^{m+1}}{m(m-1)}+ \frac{y^{m+1}+(1-y)^{m+1}-1}{m^2(m-1)}  \right] \lambda.\label{b5}
\end{align}
\end{lemma}

\begin{theorem}
The sequence
$ \bar{B}_ {n,m}\left(f;x,y;\lambda\right)$
of operators convergences uniformly to $f(x,y)$ on $I$ for each $f\in C\left(I\right)$.
\end{theorem}
\begin{proof}It is enough to prove the following condition
	\begin{eqnarray*}
	\lim_{n,m \to \infty} \bar{B}_ {n,m}\left(e_{ij}(x,y);x,y;\lambda\right)=x^i y^j, \ \ (i,j) \in \left\lbrace (0,0), (1,0), (0,1)\right\rbrace
	\end{eqnarray*}
converges uniformly on $I$.
We clearly have
\begin{eqnarray*}
	\lim_{m,n \to \infty}\bar{B}_ {n,m}\left(e_{00}(x,y);x,y;\lambda\right)=1.
\end{eqnarray*}
We have 
\begin{eqnarray*}
	\lim_{n,m \to \infty}\bar{B}_{n,m}\left(e_{10}(x,y);x,y;\lambda\right)&=&\lim_{n \to \infty}\left[ x+\frac{1-2x+x^{n+1}-(1-x)^{n+1}}{n(n-1)}\lambda\right] =e_{10}(x,y),\\
	\lim_{n,m \to \infty}\bar{B}_{n,m}\left(e_{01}(x,y);x,y;\lambda\right)&=&\lim_{m \to \infty}\left[ y+\frac{1-2y+y^{m+1}-(1-y)^{m+1}}{m(m-1)}\lambda\right]=e_{01}(x,y)
\end{eqnarray*}
by Lemma \ref{lemx}, and 
\begin{eqnarray*}
	&&\lim_{n,m \to \infty}\bar{B}_{n,m}\left(e_{02}(x,y)+e_{20}(x,y);x,y;\lambda\right)\\
&&= \lim_{n,m \to \infty}\left\lbrace x^2+\frac{x(1-x)}{n}+\left[ \frac{2x-4x^2+2x^{n+1}}{n(n-1)}+ \frac{x^{n+1}+(1-x)^{n+1}-1}{n^2(n-1)}  \right] \lambda\right. \\
&&\qquad+\left. y^2+\frac{y(1-y)}{m}+\left[ \frac{2y-4y^2+2y^{m+1}}{m(m-1)}+ \frac{y^{m+1}+(1-y)^{m+1}-1}{m^2(m-1)}  \right] \lambda\right\rbrace \\
	&&=e_{02}(x,y)+e_{20}(x,y).
\end{eqnarray*}
Bearing in mind the above conditions and
Korovkin type theorem established by Volkov \cite{volkov}
	\begin{eqnarray*}
	\lim_{m,n \to \infty}\bar{B}_{n,m}\left(e_{ij}(x,y);x,y;\lambda\right)=x^i y^j
\end{eqnarray*}
converges uniformly.
\end{proof}

Now we compute the rates of convergence of operators
$\bar{B}_ {n,m}\left(f;x,y;\lambda\right)$ to $f (x, y)$ by means of
the modulus of continuity. We first give the needed definitions.

Complete modulus of continuity for a bivariate case is defined as follows:
\begin{equation*}
\omega(f,\delta)=\sup\left\{|f(s,t)-f(x,y)|:\sqrt{(s-x)^{2}+(t-y)^{2}}\leq\delta\right\}
\end{equation*}
for $f \in C( I_{ab})$ and for every $(s,t),(x, y)\in I_{ab}=[0,a]\times[0,b]$. Partial moduli of continuity with respect to $x$ and $y$ are defined as
\begin{eqnarray*}
\omega_{1}(f,\delta)&=&\sup\left\{|f(x_{1},y)-f(x_{2},y)|:y\in[0,a]\,\,\textrm{and}\,\,|x_{1}-x_{2}|\leq\delta\right\},\\
\omega_{2}(f,\delta)&=&\sup\left\{|f(x,y_{1})-f(x,y_{2})|:x\in[0,b]\,\,\textrm{and}\,\,|y_{1}-y_{2}|\leq\delta\right\}.
\end{eqnarray*}
Peetre's $K$-functional is given by
\begin{eqnarray*}
K(f,\delta)=\inf_{g\in
C^{2}(I_{ab})}\left\{\|f-g\|_{C(I_{ab})}+\delta\|g\|_{C^{2}(I_{ab})}\right\}
\end{eqnarray*}
for $\delta> 0$, where $C^{2}(I_{ab})$ is the space of functions of $f$ such that
$f$, $\frac{\partial^{j}f}{\partial x^{j}}$ and
$\frac{\partial^{j}f}{\partial y^{j}}$ $(j=1,2)$ in $C(I_{ab})$ \cite{Peetre}. 
We now give an estimate of the rates of
convergence of operators $\bar{B}_ {n,m}\left(f;x, y;\lambda\right)$.
\begin{theorem}
\label{thorem3.1}Let $f\in C\left(I \right) $, then we have
\begin{equation*}
\left\vert\bar{B}_{n,m}\left(f;x, y;\lambda\right) -f\left(x, y\right)
\right\vert \leq 4\omega\left( f;\sqrt{\delta_{n}(x)},\sqrt{\delta_{n}(y)} \right) \label{10.2}
\end{equation*}
for all $x\in I$.
\end{theorem}

Now we investigate convergence of the sequence of linear positive
operators $\bar{B}_ {n,m}(f;x,y;\lambda)$ to  a function of
two variables which defined on weighted space.

Let $\rho(x,y)=x^{2}+y^{2}+1$  and $B_{\rho }$ be the space of all
functions defined on the real axis provided with  $|f(x,y)|\leq
M_{f}\rho (x,y),$ where $ M_{f} $ is a positive constant depending
only on $f$. 

\begin{theorem}
For each $f\in C^{0}_{\rho}$ and for all $(x; y) \in I$, we
have
\begin{eqnarray*}
\lim_{n \to \infty}\parallel \bar{B}_ {n,m}
(f;x,y;\lambda)-f(x,y)\parallel_{\rho}=0.
\end{eqnarray*}
\end{theorem}

\end{document}